\documentclass[]{article}

\usepackage{amsmath,amsfonts,amssymb,amsthm}
\usepackage[shortlabels]{enumitem}
\usepackage{mathrsfs}
\usepackage{hyperref}
\usepackage{xcolor}

\DeclareRobustCommand{\eins}{%
  \text{\usefont{U}{bbold}{m}{n}1}%
}

\usepackage{graphicx}
\newcommand{\indep}{\rotatebox[origin=c]{90}{$\models$}}

\newlength{\drop}

\theoremstyle{plain}
\newtheorem{theorem}{Theorem}[section]

\newtheorem{definition}[theorem]{Definition} 
\newtheorem{claim}[theorem]{Claim}
\newtheorem{lemma}[theorem]{Lemma}

\DeclareMathOperator*{\arginf}{arg\,inf}

\usepackage[symbol]{footmisc}
\renewcommand{\thefootnote}{\fnsymbol{footnote}}
\renewcommand*{\thefootnote}{\arabic{footnote}}
\begin{document}
  \begin{titlepage}
    \drop=0.1\textheight
    \centering
    \vspace*{\baselineskip}
    \rule{\textwidth}{1.6pt}\vspace*{-\baselineskip}\vspace*{2pt}
    \rule{\textwidth}{0.4pt}\\[\baselineskip]
    {\LARGE A Brief on Optimal Transport}\\[0.2\baselineskip]
    \rule{\textwidth}{0.4pt}\vspace*{-\baselineskip}\vspace{3.2pt}
    \rule{\textwidth}{1.6pt}\\[\baselineskip]
    \scshape
    Special Lecture\\
    Missouri S\&T : Rolla, MO
    \vspace*{2\baselineskip}\\
    Presentation by \\[\baselineskip]
    {\Large Austin G. Vandegriffe\par}
    \vfill
    {\scshape 2020} \\
  \end{titlepage}
  
\section*{Ackowledgement}
    I want to recognize my friend and colleague, Louis Steinmeister\footnote{Louis Steinmeister: \href{https://scholar.google.com/citations?user=BFSEaMkAAAAJ&hl=en}{ https://scholar.google.com/citations?user=BFSEaMkAAAAJ\&hl=en}}, for introducing and motivating my studies in pure mathematics. Without Louis, I likely would not have discovered the intricacies of measure theory, which lead me to optimal transport, for a while longer. I would also like to thank Dr. Jason C. Murphy\footnote{Jason Murphy: \href{https://scholar.google.com/citations?user=32q4x_cAAAAJ&hl=en}{https://scholar.google.com/citations?user=32q4x{\textunderscore}cAAAAJ\&hl=en}} for guiding me through some proofs and for the sanity checks; his input was of great value.
    
    \vfill

\break
 
\pagenumbering{gobble}
\renewcommand{\thefootnote}{\fnsymbol{footnote}}
\begin{center}
    \vspace*{\stretch{1}}
    \Huge{\textbf{Preliminaries}}\footnote[1]{This section has not been developed in detail as the focus of this text is optimal transport. The details are left as a future work.}
    \vspace*{\stretch{1}}
\end{center}
\renewcommand*{\thefootnote}{\arabic{footnote}}
\setcounter{footnote}{0}
\break
\pagenumbering{arabic}
  \section{Topology}
    
    \begin{definition}[ Topology, Open/Closed Sets, and Metrizablility]\label{Ch1:Topology} \label{Ch1:Metric}
        Let $\Omega$ be a non-empty set, then a collection $\tau \subseteq 2^{\Omega}$ is a \emph{topology} if
        \begin{enumerate}[i.]
            \item $\emptyset, \Omega \in \tau$
            \item $T_{1},T_{2} \in \tau ~\implies T_{1}\cap T_{2} \in \tau$
            \item  $(T_{i})_{i\in\mathcal{I}\subset\mathbb{R}} \subset \tau \implies \bigcup\limits_{i\in\mathcal{I}} T_{i} \in \tau$
        \end{enumerate}
        Sets in $\tau$ are called \emph{open}, if $T^{c}\in\tau$ then $T$ is closed, and if $T$ is both open and closed it is \emph{clopen}. If $\exists$ a metric on $\Omega$ which induces the topology, then $(\Omega,\tau)$ is called metrizable, that is, $\exists d:\Omega\times\Omega \rightarrow \mathbb{R}$ satisfying $\forall~x_{i}\in\Omega ~(i=1,2,3)$
        \begin{enumerate}[i.]
            \item $d(x_{1},x_{2}) \geqq 0$ with $d(x_{1},x_{2}) = 0 ~\iff~ x_{1}=x_{2}$
            \item $d(x_{1},x_{2}) = d(x_{2},x_{1})$
            \item $d(x_{1},x_{3}) \leqq d(x_{1},x_{2}) + d(x_{2},x_{3})$
        \end{enumerate}
        and $\tau = \{~ \{ x\in\Omega : d(x_{0},x) < r) \} : r \in \overline{\mathbb{R}}_{0}^{+} \emph{~\&~} x_{0} \in \Omega ~\}$.
    \end{definition}
    
    \begin{definition}[Hausdorff Condition]
        $(\Omega, \tau)$ is called \emph{Hausdorff} if $\forall~x_{1}\neq x_{2} \in \Omega ~\exists U_{i}\in\tau\left|_{\ni x_{i}} \right. ~s.t.~ U_{x_{1}}\bigcap U_{x_{2}} = \emptyset$.
    \end{definition}
    
    \begin{definition}[Polish Space]
        A topological space $(\Omega,\tau)$ is Polish if it is a separable, completely metrizable, topological space, that is, a space homeomorphic to a complete metric space that has a countable dense subset.
    \end{definition}
  
    \begin{definition}[Lower Semicontinuity]
        Let $(\Omega, \tau)$ be a topological space, a function $f:\Omega \rightarrow \mathbb{R}$ is lower semicontinuous if one of the following holds
        \begin{enumerate}[i.]
            \item $\{ x : f(x) > \alpha \} \in \tau ~~\forall~\alpha\in\mathbb{R}$
            \item $\{ x : f(x) \leqq \alpha \}^{c} \in \tau ~~\forall~\alpha\in\mathbb{R}$
            \item If $\tau$ is metrized by $d$, $\forall~(x_{0}\in\Omega,\epsilon > 0) ~\exists~\delta(x_{0},\epsilon)=\delta > 0$ s.t. 
            \[
                d(x_{0},x) < \delta \implies f(x) > f(x_{0}) - \epsilon
            \]
        \end{enumerate}
    \end{definition}
    
    \begin{definition}[Limit Points and Convergence]
    A point $x^{*}$ is a limit point of $A\subset\Omega$ if $\forall U\in\tau\left|_{\ni x^{*}}\right. ~A\bigcap U \neq \emptyset$.
    A sequence $(x_{n})_{n\geqq 1}$ \emph{converges} to a point $x^{*}\in\Omega$ if $\forall~ U\in\tau\left|_{\ni x^{*}}\right. ~\exists N\in\mathbb{N}$ s.t. $x_{n}\in U ~\forall n \geqq N$; if $\tau$ is metrized by $d$, then the former can be formulated as
    \[
        \forall~\epsilon>0 ~\exists N\in\mathbb{N} ~s.t.~ d(x^{*},x_{n}) < \epsilon ~\forall~ n \geqq N
    \]
    
    \end{definition}
\break
    \begin{theorem}[Equivalence of Compactness (Theorem 28.2, pg. 179 \cite{Munkres2000})]
        If a set $\mathcal{C}\subset\Omega$ is compact, then the following are equivalent
        \begin{enumerate}[i.]
            \item $\forall \left(~\mathcal{I}\subseteq\mathbb{R}, (T_{i})_{i\in\mathcal{I}}\subseteq\tau\right) ~\text{s.t.}~ \mathcal{C}\subseteq \textstyle{\bigcup\limits_{i\in\mathcal{I}}} T_{i} ~\exists (C_{n})_{n=1}^{N < \infty} \subseteq (T_{i})_{i\in\mathcal{I}} ~\text{s.t.}~ \mathcal{C} \subseteq \textstyle{\bigcup\limits_{n=1}^{N}}C_{i}$
            
            \item $\forall~ (x_{n})_{n\geqq 1}\subseteq\mathcal{C} ~\exists (x_{n_{k}})_{k\geqq1}\subseteq (x_{n})_{n\geqq 1} \emph{~\&~} x\in\mathcal{C}$ s.t. $x_{n_{k}} \xrightarrow{k\uparrow\infty} x$
            
            \item Every infinite subset of $\mathcal{C}$ has a limit point
        \end{enumerate}
    \end{theorem}
    
    \begin{theorem}[Tychonoff's Compactness (Theorem 37.3, pg. 234 \cite{Munkres2000})] \label{Ch1:Tychonoff}
        An arbitrary product of compact spaces is compact in the product topology.
    \end{theorem}

\break
  \section{Measure Theory}
	\begin{definition}[$\sigma$-Algebra] Let $\Omega$ be a non-empty set. A collection $\mathcal{F}\subset2^{\Omega}$ is a $\sigma$-algebra if
	\begin{enumerate}[i.]
	    \item $\Omega \in \mathcal{F}$
	    \item $A \in \mathcal{F} \implies A^{c} \in\mathcal{F}$
	    \item $(A_{i})_{i\geqq1}\subset\mathcal{F} \implies \displaystyle\bigcup\limits_{i\geqq1}A_{i} \in \mathcal{F}$
	\end{enumerate}
	\end{definition}

	\begin{definition}[Measurable Maps / Random Variables]
        A map $X : (\Omega, \mathcal{F}) \rightarrow (\Omega', \mathcal{F}')$ is called \emph{measurable} if
        \[
            X^{-1}(\mathcal{F}') \subset \mathcal{F}
        \]
	\end{definition}
	
	\begin{definition}[Measure]
        Let $(\Omega, \mathcal{F})$ be a measureable space, the set function $\mu:\mathcal{F}\rightarrow \mathbb{R}_{0}^{+} \cup\{\infty\}$ is a measure if
        \begin{enumerate}[i.]
            \item $\mu(\empty) = 0$
            \item $\mu(\biguplus A_{i}) = \sum\limits_{i\geqq1}\mu(A_{i})$
        \end{enumerate}
	\end{definition}
	
	\begin{definition}[Image Measure / Distribution]
        Let $X : (\Omega, \mathcal{F},\mathbb{P}) \rightarrow (\Omega', \mathcal{F}')$ be a random variable, then we can endow $(\Omega', \mathcal{F}')$ with the \emph{distribution}, often called the \emph{image measure}, $\mu=X_{\#}\mathbb{P} = \mathbb{P}\circ X^{-1}$
	\end{definition}
	
	\begin{theorem}[Eq. 1.2 pg. 18 in \cite{Villani2003}]\label{Ch1:EqualMeasuresViaBCF}
	    Let $(\Omega, \mathcal{F})$ be a measurable Polish space and $\mu_{i} \in \mathscr{P}(\mathcal{F}) ~(i=1,2)$, then $\mu_{1}=\mu_{2}$ if 
	    \[
	        \int\limits_{\Omega} \phi ~d\mu_{1} = \int\limits_{\Omega} \phi ~d\mu_{2} ~~~~\forall \phi\in\mathcal{C}_{b}^{0}(\Omega)
	    \]
	\end{theorem}
	\begin{proof}
        Assume that $\left(X,d\right)$ is a metric space and let $F$ be a closed subset of $X$. For $S\subset X$ and $x\in S$, define $d(x,S):=\inf\{d(x,y),y\in S\}$. Let $O_n:=\left\{x\in X,d(x,F)<n^{-1}\right\}$. Then $O_n$ is open and the map $$f_n\colon x\mapsto \frac{d\left(x,X\setminus O_n\right)}{d\left(x,X\setminus O_n\right)+d\left(x,F\right)}$$ is continuous and bounded. It converges pointwise and monotonically to the characteristic function of $F$. So we get by monotone convergence that $\mu(F)=\nu(F)$ for all closed set $F$. Now given a Borel set $B$ and $\varepsilon>0$, we can find a closed set $F$ and an open set $O$ such that $F\subset S\subset O$ and $\mu(O\setminus S)\leqslant\varepsilon$.
	\end{proof}

	\begin{definition}[Tightness]
        A family of finite measures $\mathcal{P}\subset\mathscr{P}_{f}(\mathcal{F})$ is called tight if
        \[
            \forall \epsilon>0 ~ \exists ~\text{compact}~ \mathcal{K}\subset\Omega ~\text{s.t.}~ \sup\limits_{\mu\in\mathcal{P}}(\mu(\Omega\setminus\mathcal{K})) < \epsilon
        \]
	\end{definition}

    \begin{lemma}[Ulam’s theorem]\label{Ch1:Ulam}
        If $(\Omega, \tau)$ is a Polish space, and $(\Omega, \mathcal{B}(\tau), \mu)$ is a probability space, then $\mu$ is tight.
    \end{lemma}
    \begin{proof}
        Result from measure theory, see Theorem 2.49 in \cite{Orbanz2016}.
    \end{proof}~\\[-2em]

    \begin{lemma}[Prokhorov’s theorem]\label{Ch1:Prokhorov}
        If $(\Omega, \tau)$ is a Polish space, then a set $\mathcal{P}\subset\mathscr{P}(\mathcal{\mathcal{B}(\tau)})$ is precompact for the weak topology if and only if it is tight.
    \end{lemma}
    \begin{proof}
        Result from measure theory, see Theorem 13.29 in \cite{Klenke2014}.
    \end{proof}~\\[-2em]
    
	\begin{definition} [Narrow / Weak$^{*}$ Convergence]
	    A sequence of measures $(\mu_{n})_{n\geqq 1}$ converges \emph{narrowly} to $\mu$ if $\forall~ f\in \mathcal{C}^{0}_{b}(\Omega)$
	    \[
	        \int f ~ d\mu_{n} \xrightarrow{n\uparrow\infty} \int f ~ d\mu
	    \]
	\end{definition}~\\[-1.5em]

\break
\addtocounter{page}{-1}
\thispagestyle{empty}
\begin{center}
    \vspace*{\stretch{1}}
    \Huge{\textbf{Introduction to Optimal Transport}}
    \vspace*{\stretch{1}}
\end{center}
\break
    \section{Coupling and Tranport Plans}
    
	\begin{definition}[Coupling]\label{def:Coupling}
        Given two probability spaces, $(\Omega_{1}, \mathcal{F}_{1}, \mu_{1})$ and $(\Omega_{2}, \mathcal{F}_{2}, \mu_{2})$, a \emph{coupling} of $(\mu_{1},\mu_{2})$, is a measure $\pi$ on $\Omega_{1}\times\Omega_{2}$ with marginals $\mu_{1}$ and $\mu_{2}$, that is,
        \begin{align*}
            & \pi(F_{1} \times \Omega_{2}) = \mu_{1}(F_{1}) ~~\forall F_{1} \in\mathcal{F}_{1} \\
            & \pi(\Omega_{1} \times F_{2}) = \mu_{2}(F_{2}) ~~\forall F_{2}\in\mathcal{F}_{2}
        \end{align*}
        Alternatively, one can view a coupling as a pair of random variables $X_{i} : (\Omega_{*}, \mathcal{F}_{*},\mathbb{P}) \rightarrow (\Omega_{i}, \mathcal{F}_{i}, \mu_{i})$ satisfying $\mu_{i} = {X_{i}}_{\#}\mathbb{P}$. Just find a pair $(X_{1},X_{2})$ satisfying $\pi = \mathbb{P}\circ(X_{1}^{-1},X_{2}^{-1})$; there are existence theorems (Theorem 1.104 in \cite{Klenke2014}) which show that, for a given distribution, there exists a random variable which generates that distribution.
	\end{definition}
	
	\begin{claim}\label{Ch2:CouplingMarginalEquation}
	    $\pi$ is a coupling of the probability measures $(\mu_{1}, \mu_{2})$ iff $\forall(\phi_{1},\phi_{2})\in L^{1}(\Omega_{1},\mu_{1})\times L^{1}(\Omega_{2},\mu_{2})$, or equivalently $L^{\infty}(\Omega_{1},\mu_{1})\times L^{\infty}(\Omega_{2},\mu_{2})$, we have
	    \begin{align*}
	        \int\limits_{\Omega_{1}\times\Omega_{2}}(\phi_{1}\oplus\phi_{2})(x,y) d\pi(x,y) = \int\limits_{\Omega_{1}} \phi_{1}(x) ~d\mu_{1}(x) + \int\limits_{\Omega_{2}} \phi_{2}(y) ~d\mu_{2}(y)
	    \end{align*}
	\end{claim}
	\begin{proof}
        ``$\implies$'' Take $\phi_{i} = \eins_{A^{(i)}}$ for some $A^{(i)}\in\mathcal{F}_{i}$, then
        \[
            \int\limits_{\Omega_{1}\times\Omega_{2}} \phi_{1} + \phi_{2} ~d\pi = \sup\limits_{(F^{(1)},F^{(2)})\in\mathcal{F}_{1}\times\mathcal{F}_{2})} (\pi(A^{(1)}, F^{(2)}) + \pi(F^{(1)}, A^{(2)}))
        \]
        by definition of Lebesgue integration. Since measures are monotone and $A^{(1)}\times F^{(2)} \subset A^{(1)}\times \Omega_{2} \in \mathcal{F}_{1}\times\mathcal{F}_{2} ~\forall F^{(2)}\in\mathcal{F}_{2}$ and similarly for $F^{(1)} \times A^{(2)}$, the integral is maximized when $(F^{(1)}, F^{(2)}) = (\Omega_{1}, \Omega_{2})$, and by the marginal restrictions of $\pi$ we have
        \[
            \int\limits_{\Omega_{1}\times\Omega_{2}} \phi_{1} + \phi_{2} ~d\pi =\mu_{1}(A^{(1)}) + \mu_{2}(A^{(2)})
        \]
        Then proceed to simple functions and then limits of simple functions.\\[1em]
        ``$\impliedby$'' Take $\phi_{i} = \eins_{F^{(i)}}$ for $F^{(i)}\in\mathcal{F}_{i}$ and $\phi_{2} = 0$, then
        \[
            \pi(F^{(1)}\times\Omega_{2}) = \int\limits_{\Omega_{1}\times\Omega_{2}} \phi_{1} ~d\pi = \int\limits_{\Omega_{1}} \phi_{1} ~d\mu_{1} = \mu(F^{(1)})
        \]
        Similarly with $\phi_{2}$ we obtain the result.
	\end{proof}
	
	\begin{definition}[Transport Plans]
	    The set of transport plans is the set of couplings on $\Omega_{1}\times\Omega_{2}$ for $(\mu_{1}, \mu_{2})$, that is,
	    \begin{align*}
	        \Pi(\mu_{1},\mu_{2}) = \{\pi : \mathcal{F}_{1} \times \mathcal{F}_{2} \rightarrow \mathbb{R}_{0}^{+} \cup\{\infty\} ~~\vert~ \pi \text{~couples~} (\mu_{1},\mu_{2})\}
	    \end{align*}
	\end{definition}
	\begin{claim}
        Given two probability spaces, $(\Omega_{1}, \mathcal{F}_{1}, \mu_{1})$ and $(\Omega_{2}, \mathcal{F}_{2}, \mu_{2})$, the set of transport plans is nonempty.
	\end{claim}
	\begin{proof}
        $(\mu_{1}\otimes\mu_{2}) \in \Pi(\mu_{1},\mu_{2})$
	\end{proof}

\break
    \section{Kantorovich' O. T. \& Basic Properties}
    Let $c : \Omega_{1}\times\Omega_{2} \rightarrow \mathbb{R}_{0}^{+} \cup\{\infty\}$ be a loss metric and $(\Omega_{i}, \mathcal{F}_{i},\mu_{i})$ a probability space. The Kantorovich optimal transport problem is finding a $\pi^{*}$ satisfying
    \[
        \pi^{*} \in \arginf\limits_{\pi\in\Pi(\mu_{1},\mu_{2})}\mathbb{K}_{c}(\pi) = \arginf\limits_{\pi\in\Pi(\mu_{1},\mu_{2})} \int\limits_{\Omega_{1}\times\Omega_{2}} c(x,y) ~d\pi(x,y)
    \]
    where the quantity `$c(x,y)~d\pi(x,y)$' can be interpreted as `` moving the amount $d\pi(x,y)$ from $x$ to $y$ at a cost $c(x,y)$.'' The minimal cost will be denoted $\mathcal{C}_{c}(\mu_{1},\mu_{2}) = \mathbb{K}_{c}(\pi^{*})$. The problem can also be posed with random variables, using the same notation as in Definition \ref{def:Coupling}, we have
    \[
        (X^{*}_{1},X^{*}_{2}) \in \arginf\limits_{\substack{X_{i}\in\mathcal{F}_{i}\\ \mu_{i}={X_{i}}_{\#}\mathbb{P}}} \mathbb{E}_{\mathbb{P}}\left[c(X_{1},X_{2})\right]
    \]
    We want to prove the following:
    \begin{theorem}[Existence of an optimal coupling]\label{CH3:OptimalCouplingExistence}
        Let $(\Omega_{i},\mathcal{F}_{i}) ~(i=1,2)$ be two Polish probability spaces, i.e. a separable, completely metrizable, topological, probability space; let $a_{i} \in L^{1}(\Omega_{i}, \mathbb{R} \cup \{-\infty\}, \mu_{i})  ~~(i = 1,2)$ be two upper semicontinuous functions. Let $c : \Omega_{1}\times\Omega_{2} \rightarrow \mathbb{R} \cup\{\infty\}$ be a lower semicontinuous cost function, such that $c(x,y) \geqq a_{1}(x) + a_{2}(y)$ for all $x, y$. Then there is a coupling of $(\mu_{1},\mu_{2})$ which minimizes the total cost $\mathbb{E}\left[~c(X_{1},X_{2})~\right]$ among all possible couplings $(X_{1},X_{2})$.
    \end{theorem}
    \begin{proof}[\textbf{Note}]
        The lower bound for $c(\bullet, \bullet)$ in Theorem \ref{CH3:OptimalCouplingExistence} guarantees a lower bound for the Kantorovich problem, this is because, by Claim \ref{Ch2:CouplingMarginalEquation},
        \[
            \inf\limits_{\pi\in\Pi(\mu_{1},\mu_{2})}\int a_{1} + a_{2} ~d\pi = \int a_{1} + a_{2} ~d(\mu_{1}\otimes\mu_{2}) \leqq \mathbb{K}_{c}(\pi^{*})
        \]
        and since the '$a_{i}$'s are integrable, we obtain a lower bound.
    \end{proof}
    \noindent To prove Theorem \ref{CH3:OptimalCouplingExistence}, we will first need a few lemmas.
	\begin{lemma}[Baire's Theorem]\label{Ch3:Baire}
		If $f$ is a nonnegative lower semicontinuous function on $\Omega$, then $\exists (f_{n})_{n\geqq 1} \subset \mathcal{C}^{0}(\Omega)$ such that $f_{n}\uparrow f$ pointwise.
	\end{lemma}
	\begin{proof}
	    Let
	    \[
	        f_{n}(\bullet) = \inf\limits_{z\in\Omega}(f(z) + n\cdot d( \bullet ,z))
	    \]
	    Then all we must show is that $f_{n}$ is (i) increasing, (ii) continuous, (iii) convergent to $f$.
	    \begin{enumerate}[(i)]
	        \item We have the inequality $f(z) + n\cdot d(x,z) \leqq f(z) + (n+1)\cdot d(x,z)$. ~Now taking the $\inf$ in $z$ on both sides yields
	        \[
                    \inf\limits_{z\in\Omega}(f(z) + n\cdot d(x,z)) = f_{n}(x) \leqq f_{n+1}(x) = \inf\limits_{z\in\Omega}(f(z) + (n+1)\cdot d(x,z))
	       \]
	       Also note that $z$ can be $x$, so we have $d(x,x) = 0$ and hence
	       \[
	            f_{n} \leqq f \tag{$\star$}\label{fnleqf}
	       \]
	       \item Choose $x,y,z \in \Omega$, then, by the triangle inequality, we have
	       \[
                    f(z) + n\cdot d(z,x) \leqq f(z) + n\cdot d(z,y) + n\cdot d(y,x)
	       \]
	       Taking the $\inf$ in $z$ on both sides and subtracting $f_{n}(y)$ from both sides
	       \[
                    f_{n}(x) - f_{n}(y) \leqq n\cdot d(y,x)
	       \]
	       Since $x$ and $y$ were arbitrary one obtains
	       \[
                    \left| f_{n}(x) - f_{n}(y) \right| \leqq n\cdot d(x,y)
	       \]
	       Now let $\epsilon > 0$ and $\delta = \dfrac{\epsilon}{n}$, then we have
	       \[
                    d(x, y) < \delta \implies \left| f_{n}(x) - f_{n}(y) \right| \leqq n\cdot d(y,x) < n\cdot \left( \dfrac{\epsilon}{n} \right) = \epsilon ~~\forall x,y
	       \]
	       and we have uniform continuity.
	       
	        \item Lower semicontinuity yields, for a fixed $x_{0}\in\Omega$, that $\forall\epsilon > 0$ $\exists~ \delta > 0 $ such that 
	        \[
                    d(x_{0}, z) < \delta \implies f(z) > f(x_{0}) - \epsilon \tag{$\star\star$}\label{lsc}
	       \]
	       Now suppose $d(x_{0}, z) > \delta$, then $\exists~ N \in \mathbb{N}$ such that $\forall~ n \geqq N$
	       \[
	            f(z) + n\cdot d(x_{0},z) \geqq n\cdot\delta > f(x_{0})
	       \]
	       since $f > 0$. However, we have, from \eqref{fnleqf}, that $f_{n} \leqq f ~\forall~n\in\mathbb{N}$, so
	       \[
	            d(~x_{0}~,~\arginf\limits_{z\in\Omega}\left(f(z) + n\cdot d(x_{0}, z))~\right) < \delta
	       \]
	       But, for all $z\in\Omega$ s.t. $d(x_{0},x) < \delta$ we have, from \eqref{lsc},
	       \[
	            f(z) + n\cdot d(x_{0},z) \geqq f(z) > f(x_{0}) - \epsilon
	       \]
	       and so
	       \[
	            \inf\limits_{\substack{z\in\Omega\\d(x_{0},z)<\delta}}( f(z) + n\cdot d(x_{0},z)) = f_{n}(x_{0}) > f(x_{0}) - \epsilon
	       \]
	       Therefore,
	       \[
	            f(x_{0}) - \epsilon \leqq f_{n}(x_{0}) \leqq f(x_{0})
	       \]
	       and so, taking $\epsilon\downarrow 0$ and $n\uparrow\infty$, we have $f_{n}(x_{0}) \rightarrow f(x_{0})$. Since $x_{0} \in \Omega$ was arbitrary, we have that $f_{n} \rightarrow f$ pointwise.
	    \end{enumerate}
	\end{proof}
	
    \begin{lemma}
		Let $f$ be a nonnegative lower semicontinuous function on $\Omega$. If $(\mu_{n})_{n\geqq 1}$ converges narrowly to $\mu$, then
		\[
		    \int f ~ d\mu \leqq \liminf\limits_{n\uparrow\infty} \int f ~ d\mu_{n}
		\]
	\end{lemma}
	\begin{proof}
	    Since $g$ is lower semicontinuous $\exists (f_{n})_{n\geqq 1} \subset \mathcal{C}_{b}^{0}(\Omega)$ such that $f_{n} \uparrow f$ by Lemma \ref{Ch3:Baire}. Let $\epsilon > 0$, by the Beppo-Levi lemma for nonnegative increasing measurable functions, we have that $\exists K\in\mathbb{N}$ s.t. $\forall k>K$
	    \[
	        \left|\int f ~d\mu - \int f_{k} ~d\mu\right| = \int f ~d\mu - \int f_{k} ~d\mu \leqq \epsilon
	    \]
	    rearranging we obtain
	    \[
	        \int f ~d\mu \leqq \int f_{k} ~d\mu + \epsilon \tag{$\star$}\label{Ch3:3-3-BeppoLevi}
	    \]
	    Now, by narrow convergence, we have that $\exists N\in\mathbb{N}$ s.t. $\forall n\geqq N$
	    \[
	        \left| \int f_{k} ~d\mu - \int f_{k} ~d\mu_{n} \right| < \epsilon \\
	    \]
	    by definition of the absolute value we have
	    \[
	         -\epsilon < \int f_{k} ~d\mu - \int f_{k} ~d\mu_{n} < \epsilon\\
	    \]
	    and then we add $\epsilon$ to both sides
	    \[
	        0 < \int f_{k} ~d\mu - \int f_{k} ~d\mu_{n} < \int f_{k} ~d\mu - \int f_{k} ~d\mu_{n} + \epsilon < 2\epsilon \tag{$\star\star$}\label{Ch3:3-3-Narrow}
	    \]
	    ``Adding zero'' to \eqref{Ch3:3-3-BeppoLevi}, applying \eqref{Ch3:3-3-Narrow}, then recalling that $\int f_{k} ~d\mu_{n} \leqq \int f ~d\mu_{n}$ since $f_{k}\leqq f ~\forall~ k\in\mathbb{N}$ we obtain
	    \begin{align*}
	        \int f ~d\mu 
	            &\leqq \int f_{k} ~d\mu \pm \int f_{k} ~d\mu_{n} + \epsilon\\
	            &= \int f_{k} ~d\mu_{n} + \left(\int f_{k} ~d\mu - \int f_{k} ~d\mu_{n}\right) + \epsilon\\
	            &\leqq \int f_{k} ~d\mu_{n} + 3\epsilon\\
	            &\leqq \int f ~d\mu_{n} + 3\epsilon
	    \end{align*}
	    Now, by taking the $\liminf$ in $n$ and taking $\epsilon\downarrow 0$ we obtain the result
	    \[
	        \liminf\limits_{n\uparrow\infty}\int f ~d\mu = \int f ~d\mu \leqq \liminf\limits_{n\uparrow\infty}\int f ~d\mu_{n}
	    \]
	\end{proof}
	
	\begin{lemma}[Lower semicontinuity of the cost functional] \label{Ch3:LSCCostFunctional}
	Let $\Omega_{1}$ and $\Omega_2$ be two Polish spaces, and $c : \Omega_{1}\times\Omega_{2} \rightarrow \mathbb{R}\cup\{\infty\}$ a lower semicontinuous cost function. Let $h : \Omega_{1}\times\Omega_{2} \rightarrow \mathbb{R}\cup\{-\infty\}$ be an upper semicontinuous function such that $c \geqq h$ everywhere. Let $(\pi_{k})_{k\in\mathbb{N}}$ be a sequence of probability measures on $\Omega_{1}\times\Omega_{2}$, converging weakly to some $\pi \in \mathscr{P}(\Omega_{1}\times\Omega_{2})$, in such a way that $h\in L^{1}(\pi_{k})\cap L^{1}(\pi)$, and
	\[
	    \int\limits_{\Omega_{1}\times\Omega_{2}} h ~d\pi_{k} \xrightarrow{k \uparrow \infty} \int\limits_{\Omega_{1}\times\Omega_{2}} h ~d\pi
	\]
	Then 
	\[
	    \int\limits_{\Omega_{1}\times\Omega_{2}} c ~d\pi \leqq \liminf\limits_{k\uparrow\infty}\int\limits_{\Omega_{1}\times\Omega_{2}} c ~d\pi_{k}
	\]
	\end{lemma}
	\begin{proof}
        Replace $c$ by $c-h$, a non-negative lower semicontinuous function, and apply the previous lemma.
	\end{proof}
	
    \begin{lemma}[Tightness of transference plans]\label{CH3:TightnessOfTransferencePlans}
    Let $\Omega_{1}$ and $\Omega_{2}$ be two Polish spaces. Let $\mathcal{P}_{1}\subset\mathscr{P}(\mathcal{F}_{1})$ and $\mathcal{P}_{2}\subset\mathscr{P}(\mathcal{F}_{2})$ be tight subsets of $\mathscr{P}(\mathcal{F}_{1})$ and $\mathscr{P}(\mathcal{F}_{2})$ respectively. Then the set $\Pi(\mathcal{P}_{1},\mathcal{P}_{1})$ of all transference plans whose marginals lie in $\mathcal{P}_{1}$ and $\mathcal{P}_{2}$ respectively, is itself tight in $\mathscr{P}(\mathcal{F}_{1}\times\mathcal{F}_{2})$.
    \end{lemma}
    \begin{proof}
        Let $\mu_{1} \in \mathcal{P}_{1}$, $\mu_{2} \in \mathcal{P}_{1}$, and $\pi \in \Pi(\mu_{1}, \mu_{2})$. By Ulam's tightness theorem (and by assumption), we have that
        \[
            \forall \epsilon > 0 ~ \exists ~\text{compact}~ \mathcal{K}^{(i)}_{\epsilon}\subset\Omega_{i} ~(\indep ~\mu_{i}) ~\text{s.t.}~ \mu_{i}(\Omega_{i}\setminus\mathcal{K}^{(i)}_{\epsilon}) \leqq \epsilon
        \]
        Let $(X_{1},X_{2})$ be a coupling of $(\mu_{1}, \mu_{2})$, that is, $law(X_{i}) = {X_{i}}_{\#}\mathbb{P} = \mu_{i}$, then
        \begin{align*}
            &\implies \mathbb{P}((X_{1},X_{2}) \not\in \mathcal{K}^{(1)}_{\epsilon}\hspace{-0.25em}\times\mathcal{K}^{(2)}_{\epsilon})\\
            &~~~~~~~=\mathbb{P}(\{~\omega ~\vert~ X_{1}(\omega)\not\in\mathcal{K}^{(1)}_{\epsilon}\} \cap \{~\omega ~\vert~ X_{2}(\omega)\not\in\mathcal{K}^{(2)}_{\epsilon}\}) \\
            &~~~~~~~=\mathbb{P}(X_{1}^{-1}(\Omega_{1}\setminus\mathcal{K}^{(1)}_{\epsilon}) \cap X_{2}^{-1}(\Omega_{2}\setminus\mathcal{K}^{(2)}_{\epsilon})) \\
            &~~~~~~~\leqq \mathbb{P}(X_{1}^{-1}(\Omega_{1}\setminus\mathcal{K}^{(1)}_{\epsilon}) \cup X_{2}^{-1}(\Omega_{2}\setminus\mathcal{K}^{(2)}_{\epsilon})) \\
            &~~~~~~~\leqq \mathbb{P}(X_{1}^{-1}(\Omega_{1}\setminus\mathcal{K}^{(1)}_{\epsilon})) +  \mathbb{P}(X_{2}^{-1}(\Omega_{2}\setminus\mathcal{K}^{(2)}_{\epsilon})) \\
            &~~~~~~~=\mu_{1}(\Omega_{1}\setminus\mathcal{K}^{(1)}_{\epsilon})) + \mu_{2}(\Omega_{2}\setminus\mathcal{K}^{(2)}_{\epsilon})) \\
            &~~~~~~~\leqq \epsilon + \epsilon = 2\epsilon ~\indep~ \mu_{i} \\
        \end{align*}~\\[-2.5em]
        And since $\mathcal{K}_{\epsilon}^{(i)}$ is compact, we have, by Theorem \ref{Ch1:Tychonoff} (Tychonoff), that $\mathcal{K}_{\epsilon}^{(1)}\hspace{-0.25em} \times \mathcal{K}_{\epsilon}^{(2)} \subset \Omega_{1}\times\Omega_{2}$ too is compact; therefore, $\Pi(\mathcal{P}_{1},\mathcal{P}_{2})$ is tight.
    \end{proof}
    
    \begin{proof}[Proof of Theorem \ref{CH3:OptimalCouplingExistence}]
        Since $\Omega_{i}$ is a Polish space, we have, by Theorem \ref{Ch1:Ulam} (Ulam), that $\mu_{i}$ is tight and, by Lemma \ref{CH3:TightnessOfTransferencePlans}, that $\Pi(\mu_{1},\mu_{2})$ is tight, and so by Theorem \ref{Ch1:Prokhorov} (Prokhorov) this set has a compact closure. Now, take $(\pi_{k})_{k\geqq 1} \subset \Pi(\mu_{1},\mu_{2})$ s.t. $\pi_{k} \xrightarrow{k\uparrow\infty} \pi$ in the narrow sense; we want to show that $\pi\in\Pi(\mu_{1},\mu_{2})$, i.e. has margins $\mu_{1}$ and $\mu_{2}$. Let $(\phi_{1},\phi_{2}) \in L^{1}(\mu_{1})\times L^{1}(\mu_{2})$, then we have
        \[
            \int\limits_{\Omega_{1}\times\Omega_{2}} \phi_{1} + \phi_{2} ~d\pi_{k} = \int\limits_{\Omega_{1}} \phi_{1} ~d\mu_{1} + \int\limits_{\Omega_{2}} \phi_{2} ~d\mu_{2}
            \xrightarrow{k\uparrow\infty} \int\limits_{\Omega_{1}\times\Omega_{2}} \phi_{1} + \phi_{2} ~d\pi
        \]
        hence
        \[
            \int\limits_{\Omega_{1}\times\Omega_{2}} \phi_{1} + \phi_{2} ~d\pi = \int\limits_{\Omega_{1}} \phi_{1} ~d\mu_{1} + \int\limits_{\Omega_{2}} \phi_{2} ~d\mu_{2}
        \]
        and by Claim \ref{Ch2:CouplingMarginalEquation} we conclude that $\pi\in\Pi(\mu_{1},\mu_{2})$; hence, $\Pi(\mu_{1},\mu_{2})$ is closed, and since it has a compact closure, we have that it is compact. Now let $(\pi_{k})_{k\geqq1}$ be the minimizing sequence for $\int c ~d\pi_{k}$ which converges to the optimal transport cost. Since $\Pi(\mu_{1},\mu_{2})$ is compact, take a narrowly convergent subsequence to $\pi\in \Pi(\mu_{1},\mu_{2})$. Notice that
        \[
            h : \Omega_{1}\times\Omega_{2} \ni (x_{1}, x_{2}) \mapsto a_{1}(x_{1}) + a_{2}(x_{2}) = h(x_1, x_2) \in \mathbb{R}
        \]
        is $L^{1}(\pi_{k}) \cap L^{1}(\pi)$ and, by assumption, $c \geqq h$ everywhere; moreover,
        \[
            \int h ~d\pi_{k} = \int h ~d\pi = \int a_{1} ~d\mu_{1} + \int a_{2} ~d\mu_{2}
        \]
        Therefore, with Lemma \ref{Ch3:LSCCostFunctional} on $c-h$, we have
        \[
    	    \int c ~d\pi \leqq \liminf\limits_{k\uparrow\infty}\int c ~d\pi_{k}
        \]
        thus $\pi$ is minimizing.
    \end{proof}

    \begin{theorem}[Optimality is inherited by restriction]
    Let $(\Omega_{i},\mathcal{F}_{i},\mu_{i}) ~(i=1,2)$ be two Polish spaces, $a_{i}\in L^{1}(\Omega_{i},\mu_{i})$, and let $c:\Omega_{1}\times\Omega_{2} \rightarrow \mathbb{R}\cup\{\infty\}$ be a measurable cost function such that $c \geqq a_{1} + a_{2}$; let $\mathcal{C}_{c}(\mu_{1},\mu_{2})$ be the optimal transport cost from $\mu_{1}$ to $\mu_{2}$. Assume $\mathcal{C}_{c}(\mu_{1},\mu_{2}) < \infty$ and let $\pi\in\Pi(\mu_{1},\mu_{2})$ be an optimal transport plan. Let $\tilde\pi$ be a nonnegative measure on $\mathcal{F}_{1}\times\mathcal{F}_{2}$ such that $\tilde\pi \leqq \pi$ and $\pi(\Omega_{1}\times\Omega_{2}) > 0$. Then the probability measure
    \[
        \pi' = \dfrac{\tilde\pi}{\tilde\pi(\Omega_{1}\times\Omega_{2})} = \dfrac{\tilde\pi}{\tilde{Z}}
    \]
    is an optimal transference plan between its marginals $\mu_{1}'$ and $\mu_{2}'$.
    
    Moreover, if $\pi$ is the unique optimal transference plan between $\mu_{1}$ and $\mu_{2}$, then $\pi'$ is the unique optimal transference plan between $\mu_{1}'$ and $\mu_{2}'$.
    \end{theorem}
    
    \begin{proof}
        Suppose $\pi'$ is not optimal, then $\exists ~\pi''$ such that, for all $F^{(i)}\in\mathcal{F}_{i}$,
        \[
            \pi''(\bullet\times F^{(2)}) = \mu_{1}', ~~~~~~~\pi''(F^{(1)}\times\bullet) = \mu_{2}'
        \]
        and
        \[
            \int c ~d\pi'' < \int c ~d\pi'
        \]
        Now, consider
        \begin{align*}
            \hat\pi &= (\pi - \tilde\pi) + \tilde{Z}\pi''\\
                &= (\pi - \tilde{Z}\dfrac{\tilde{\pi}}{\tilde{Z}}) + \tilde{Z}\pi''\\
                &= (\pi - \tilde{Z}\pi') + \tilde{Z}\pi''\\
                &= \pi + \tilde{Z}(\pi'' - \pi')
        \end{align*}
        where $\tilde{Z}=\tilde{\pi}(\Omega_{1}\times\Omega_{2}) > 0$ by assumption. It is clear that $\hat\pi$ is nonnegative since $\tilde{\pi} \leqq \pi$ and $\pi'' \geqq 0$. Note that $\hat\pi\in\Pi(\mu_{1},\mu_{2})$, that is, for all $F^{(i)}\in\mathcal{F}_{i}$
        \[
        \begin{cases}
            \hat\pi(F^{(1)}\times\Omega_{2}) = \mu_{1}(F^{(1)} + \tilde{Z}\left(~\mu_{1}'(F^{(1)}) - \mu_{1}'(F^{(1)})~\right) = \mu_{1}(F^{(1)})\\
            ~\\
            \hat\pi(\Omega_{1}\times F^{(2)}) = \mu_{1}(F^{(2)} + \tilde{Z}\left(~\mu_{2}'(F^{(2)}) - \mu_{2}'(F^{(2)})~\right) = \mu_{2}(F^{(2)})\\
        \end{cases}
        \]
        Since $\int c ~d(\tilde{Z}(\pi''-\pi')) < 0$, we obtain
        \[
            \int c ~d\hat\pi = \int c ~d\pi + \int c ~d(\tilde{Z}(\pi''-\pi'))  < \int c ~d\pi
        \]
        which contracts the optimally of $\pi$; therefore, $\pi'$ is optimal. Now suppose $\pi$ is a unique optimal transference plan, let $\pi'$ and $\pi''\in\Pi(\mu_{1}',\mu_{2}')$ be optimal, define $\hat\pi$ as above and note that $\hat\pi \leqq \pi$ (since $\pi$ is optimal), hence $\hat\pi = \pi$ yielding
        \[
                \int \phi ~d\hat\pi = \int \phi ~d\pi + \int \phi ~d(\tilde{Z}(\pi''-\pi'))  = \int \phi ~d\pi
        \]
        $\forall~ \phi\in\L^{1}(\pi)$, and so, by Claim \ref{Ch1:EqualMeasuresViaBCF} (bounded continuous equality, \textbf{not} $L^{1}(\pi)$, suppose $\infty$ at a point where $\pi \neq 0$ and $\pi'=0$) on $(\mu_{1}',\mu_{2}')$,  $\pi'=\pi''$ and so $\pi'$ is unique.
    \end{proof}

\break  
    \section{The Wasserstein Distance}
    We want to be able to say that
    \[
        \mathcal{C}_{c}(\mu_{1},\mu_{2}) = \inf\limits_{\pi\in\Pi(\mu_{1},\mu_{2})}\mathbb{K}_{c}(\pi)
    \]
    is the ``distance between $\mu_{1} ~\text{and}~ \mu_{2}$'', but, in general, $\mathcal{C}_{c}(\bullet,\bullet)$ does not satisfy the axioms of a distance function, i.e. a metric; however, we obtain such a metric characteristic when $c$ is a metric such as $\ell^{p}$ for some $p\in\mathbb{N}$.
    \begin{definition}
        Let $(\Omega,d)$ be a Polish metric space, and let $p\in [1,\infty)$. For any two probability measures $\mu_{1},\mu_{2}$ in $(\Omega,\mathcal{F})$, the \emph{Wasserstein distance of order p} between $\mu_{1}$ and $\mu_{2}$ is defined by the formula
        \begin{align*}
            \mathcal{W}_{p}(\mu_{1},\mu_{2}) 
                &= \mathcal{C}_{d^{p}(\bullet,\bullet)} ^{\frac{1}{p}}(\mu_{1},\mu_{2}) = \left( \inf\limits_{\pi\in\Pi(\mu_{1},\mu_{2})} \int\limits_{\Omega} d^{p}(x,y) ~d\pi \right)^{\frac{1}{p}}\\
                &= \inf\limits_{\substack{X_{i}\in\mathcal{F}\\law(X_{i})=\mu_{i}}} \left(\mathbb{E}\left[ d^{p}(X_{1},X_{2}) \right]^{\frac{1}{p}} \right)
        \end{align*}
    \end{definition}
    
    
    \begin{theorem}[Hahn-Banach Extension]
        Let $(\Theta, ||\bullet||)$ be a normed linear space, let $\Theta'\subset\Theta$ be a linear subspace and let $\ell \in (\Theta')^{*}$, then $\exists~ \tilde\ell \in \Theta^{*}$ such that $\ell(\omega) = \tilde\ell(\omega) ~\forall~ \omega\in\Theta'$.
    \end{theorem}
    
    \begin{definition}[Riesz Space]
        A Riesz space $\mathcal{R}$ is a vector space endowed with a partial order, that is, $\forall~ x_{1},x_{2},x_{3}\in\mathcal{R}$
        \begin{enumerate}[i.]
            \item $x_{1} \leqq x_{2} \implies x_{1}+x_{3} \leqq x_{2} + x_{3}$
            \item $\forall \alpha\geqq0$, $x_{1}\leqq x_{2} \implies \alpha x_{1}\leqq \alpha x_{2}$
            \item $\exists~\sup(x_{1},x_{2})$
        \end{enumerate}
    \end{definition}
    
    \begin{definition}[Positive Linear Functional \& Sublinear]
        $\ell \in \Omega^{*}$ is \emph{positive linear} if 
        \[
            \forall ~\omega \geqq 0 \implies \ell(\omega) \geqq 0
        \]
        and is \emph{sublinear} if $\forall \omega_{1}, \omega_{2}\in\Omega$ and $\alpha\in\mathbb{R}_{0}^{+}$
        \begin{enumerate}[i.]
            \item $\ell(\omega_{1} + \omega_{2}) \leqq \ell(\omega_{1}) + \ell(\omega_{2})$
            \item $\ell(\alpha\omega_{1}) = \alpha\ell(\omega_{1})$
        \end{enumerate}
    \end{definition}
    
    \begin{theorem}[Hahn-Banach Positive Extension (\cite{Alipratis-Border2006}, Theorem 8.31, pg. 330)]\label{Hahn-Banach-Positive}
        Let $\Theta$ be a Riesz space, $\Theta'\subset\Theta$ be a Riesz subspace, and let $\ell\in (\Theta')^{*}$ be \textbf{positive linear}, then $\ell$ extends to a \textbf{positive linear functional} on all of $\Theta$ if and only if there is a monotone sublinear functional $\rho\in \Theta^{*}$ satisfying $\ell(\theta') \leqq \rho(\theta') ~\forall~ \theta'\in\Theta'$.
    \end{theorem}
    \begin{proof}
        ``$\implies$'': Let $\tilde\ell \in \Theta^{*}$ extend $\ell$ and set $\rho(\theta) = \hat\ell(\theta^{+}) = \hat\ell(\eins(\theta\geqq0)\cdot\theta)$
        
        \noindent``$\impliedby$'': Suppose $\rho:\Theta \rightarrow \mathbb{R}$ is monotone sublinear with $\ell(\theta') \leqq \rho(\theta') ~\forall~ \theta'\in\Theta'$. By the Hahn-Banach Extension Theorem $\exists~ \hat\ell \in \Theta^{*}$ extending $\ell$ and satisfying $\ell(\theta) \leqq \rho(\theta) ~\forall~ \theta\in\Theta$. Then for $\theta \geqq 0$
        \[
            -\hat\ell(\theta) = \hat\ell(-\theta) \leqq \rho(-\theta) \leqq \rho(0) = 0
        \]
        and multiplying both sides by $-1$ we obtain
        \[
            \theta\geqq0 \implies \hat\ell(\theta) \geqq 0
        \]
        hence $\hat\ell$ is a positive extension of $\ell$.
    \end{proof}
    
    \begin{theorem}[Riesz Representation (\cite{Gaans2002}, Theorem 7.3, pg. 22)] \label{RieszRepresentationMeasure}
        Let $(\Theta, d)$ be a metric space, then $\forall~\textbf{positive}~ \ell\in\mathcal{C}_{b}^{0}(\Theta)^{*} ~ \exists~ \text{tight}~ \mu\in\mathcal{P}(\mathcal{B}(\Theta,d))$ s.t.
        \[
            \ell(f) = \int\limits_{\Theta} f ~d\mu ~~\forall~ f \in \mathcal{C}_{b}(\Theta)
        \]
    \end{theorem}
    
    \noindent{\color{red} Note:} In the following I add a \emph{compactness} assumption to $\Omega$ which has not been there up to now. As I work through \cite{Villani2009} and build the presentation for a more general approach to optimal transport, I will add similar theorems with weaker assumptions.
    
    \begin{theorem}[$\mathcal{W}_{p}$ is a metric on $\mathscr{P}(\mathcal{F})$]\label{Ch3:WassersteinAsMetric}
        Let $(\Omega, \mathcal{F}, d)$ be a measurable compact Polish metric space, then $\mathcal{W}_{p}$ is a metric on $\mathscr{P}(\mathcal{F})$.
    \end{theorem}
    
    \begin{lemma}[Gluing Lemma]\label{Ch4:Gluing}
        Let $(\Omega_{i}, \mathcal{F}_{i}, \mu_{i}) ~(i=1,2,3)$ be a compact measured Polish space with associated transport plans $\pi_{12}\in\Pi(\mu_{1},\mu_{2})$ and $\pi_{23}\in\Pi(\mu_{2},\mu_{3})$, then $\exists~ \pi_{123} \in \mathscr{P}(\mathcal{F}_{1}\times\mathcal{F}_{2}\times\mathcal{F}_{3})$ with marginals $\pi_{12}$ and $\pi_{23}$.
    \end{lemma}
    \begin{proof}
        Let $V\subset C^{0}_{b}(\Omega_{1}\times\Omega_{2}\times\Omega_{3})$ be the vector space
        \[
            V = \{ \phi_{12}(x_{1},x_{2})+\phi_{23}(x_{2},x_{3}) : \phi_{12} \in C^{0}_{b}(\Omega_{1}\times\Omega_{2}), \phi_{23} \in C^{0}_{b}(\Omega_{2}\times\Omega_{3}) \}
        \]
        and define a functions $G : V \rightarrow \mathbb{R}$ by
        \[
            G(\phi_{12}+\phi_{23}) = \int\limits_{\Omega\times\Omega} \phi_{12} ~d\pi_{12} + \int\limits_{\Omega\times\Omega} \phi_{23} ~d\pi_{23}
        \]
        We will now show that $G$ is well defined. Let $\phi_{12} + \phi_{23} = \hat\phi_{12} + \hat\phi_{23}$, lets perturb $x_{1}$ by $\Delta x_{1}$
        \begin{align*}
          \phi_{12}(x_{1}+\Delta x_{1}, x_{2}) - \tilde\phi_{12}(x_{1}+\Delta x_{1}, x_{2}) 
            &= \tilde\phi_{23}(x_{2},x_{3}) - \phi_{23}(x_{2},x_{3})\\
            &= \phi_{12}(x_{1}, x_{2}) - \tilde\phi_{12}(x_{1}, x_{2})
        \end{align*}
        and similarly for $x_{3}$ we obtain, with the equality restriction, that $\phi_{12}(x_{1}, x_{2}) - \tilde\phi_{12}(x_{1}, x_{2})$ and $\tilde\phi_{23}(x_{2},x_{3}) - \phi_{23}(x_{2},x_{3})$ are functions of $x_{2}$. Thus
        \begin{align*}
            \int\limits_{\Omega\times\Omega} \phi_{12} - \tilde\phi_{12} ~d\pi_{12} 
            &= \int\limits_{\Omega} \phi_{12} - \tilde\phi_{12} ~d\mu_{2}\\
            &= \int\limits_{\Omega} \phi_{23} - \tilde\phi_{23} ~d\mu_{2}\\
            &= \int\limits_{\Omega\times\Omega} \phi_{23} - \tilde\phi_{23} ~d\pi_{23}\\
        \end{align*}
        and, by rearranging, we obtain
        \[
            G(\phi_{12}+\phi_{23}) = G(\tilde\phi_{12}+\tilde\phi_{23}) 
        \]
        and so $G$ is well defined. Clearly $G$ is bounded and linear (its an integral), we must show that it is \emph{positive} linear. Let
        \[
            \phi_{12}(x_{1},x_{2})+\phi_{23}(x_{2}.x_{3}) \geqq 0
        \] 
        then, with both sides being functions of $x_{2}$, we have
        \[
            \phi_{12}(x_{1},x_{2}) \geqq -\phi_{23}(x_{2},x_{3}) \geqq -\inf\limits_{x_{3}}\phi_{23}(x_{2},x_{3})
        \]
        and that the infimum exists since $\phi_{23}$ is bounded. We now have
        \[
        \begin{cases}
            \displaystyle\int\limits_{\Omega_{1}\times\Omega_{2}} \phi_{12} ~d\pi_{12} \geqq \displaystyle\int\limits_{\Omega_{1}\times\Omega_{2}} -\inf\limits_{x_{3}}\phi_{23}(x_{2},x_{3}) ~d\pi_{12} \geqq -\displaystyle\int\limits_{\Omega_{2}} \inf\limits_{x_{3}}\phi_{23}(x_{2},x_{3}) ~d\mu_{2}\\
            
            ~\\
            
            \displaystyle\int\limits_{\Omega_{2}\times\Omega_{3}} \phi_{23} ~d\pi_{23} \geqq \displaystyle\int\limits_{\Omega_{2}\times\Omega_{3}} \inf\limits_{x_{3}}\phi_{23}(x_{2},x_{3}) ~d\pi_{23} \geqq \displaystyle\int\limits_{\Omega_{2}} \inf\limits_{x_{3}}\phi_{23}(x_{2},x_{3}) ~d\mu_{2}
        \end{cases}
        \]
        Using the above lower bounds, we obtain
        \begin{align*}
            G(\phi_{12}+\phi_{23}) 
                &= \int\limits_{\Omega_{1}\times\Omega_{2}} \phi_{12} ~d\pi_{12} + \int\limits_{\Omega_{2}\times\Omega_{3}} \phi_{23} ~d\pi_{23}\\
                &\geqq -\displaystyle\int\limits_{\Omega_{2}} \inf\limits_{x_{3}}\phi_{23}(x_{2},x_{3}) ~d\mu_{2} + \displaystyle\int\limits_{\Omega_{2}} \inf\limits_{x_{3}}\phi_{23}(x_{2},x_{3}) ~d\mu_{2}\\
                &= 0
        \end{align*}
        Thus, $G$ is a positive linear functional. So, by Theorem \ref{Hahn-Banach-Positive} (Hahn-Banach Positive Extension)[with $\Theta = \mathcal{C}_{b}^{0}(\Omega_{1}\times\Omega_{2}\times\Omega_{3})$ and $\Theta' = V$, and $\rho(\bullet)=\sup\limits_{x\in\Omega}(\bullet)$ in the definition], $\exists~ \hat G: \mathcal{C}_{b}^{0}(\Omega_{1}\times\Omega_{2}\times\Omega_{3}) \rightarrow \mathbb{R}$, and by the Theorem \ref{RieszRepresentationMeasure} (Riesz representation) we have $\exists~ \pi_{123}\in\mathscr{P}(\mathcal{F}_{1}\times\mathcal{F}_{2}\times\mathcal{F}_{3})$ corresponding to $\hat{G}$ yielding
        \begin{align*}
            \int\limits_{\Omega^{\otimes3}} \phi_{12} + \phi_{23} ~d\pi_{123} 
            &= \hat{G}(\phi_{12}+\phi_{23})\\
            &= G(\phi_{12}+\phi_{23})\\
            &= \int\limits_{\Omega\times\Omega} \phi_{12} ~d\pi_{12} + \int\limits_{\Omega\times\Omega} \phi_{23} ~d\pi_{23} ~~~~\forall~ (\phi_{12}+\phi_{23}) \in V
        \end{align*}
        and we have, from Theorem \ref{Ch1:EqualMeasuresViaBCF}, that $\pi_{123}$ has marginals $\pi_{12}$ and $\pi_{23}$ as desired.
    \end{proof}
    
    \begin{proof}[Proof of Theorem \ref{Ch3:WassersteinAsMetric}]
        We must show that $\mathcal{W}_{p}$ satisfies the properties of a metric in Definition \ref{Ch1:Metric}. It is clear that $\mathcal{W}_{p}$ is non-negative, symmetric, and finite (since the infimum is achieved). Now, suppose $\mu_{1}=\mu_{2}=\mu$, then there exists a random variable $X: (\Omega_{*},\mathcal{F}_{*},\mathbb{P}) \rightarrow (\Omega,\mathcal{F})$ such that $\mu=X_{\#}\mathbb{P}$, then, with $X_{1}=X_{2}=X$ in the definition of the Kantorovich problem, we obtain
        \[
            \int\limits_{\Omega_{*}} d^{p}(X(\omega),X(\omega)) ~d\mathbb{P}(\omega) = 0
        \]
        so $\mathcal{W}_{p}(\mu,\mu) = 0 ~\forall~\mu\in\mathscr{P}(\mathcal{F})$. Now let $\mu_{1}, \mu_{2} \in \mathscr{P}(\mathcal{F})$ (not necessarily equal), if $\mathcal{W}_{p}(\mu_{1},\mu_{2}) = 0$, then $\pi^{*}$ must concentrate all of its mass on the diagonal $\Delta_{\Omega} \subset \Omega\times\Omega$; suppose it didn't, then $\exists F \in \mathcal{F}\times\mathcal{F}\left|_{ \Delta_{\Omega}^{c}}\right.$ s.t. $\pi^{*}(F) > 0$ and $\sup\limits_{(x_{1},x_{2})\in F}(d(x_{1},x_{2})) > 0$, so we have 
        \[
            \left(\int\limits_{F}d^{p}(x_{1},x_{2}) ~d\pi \right)^{\frac{1}{p}} \leqq \mathcal{W}_{p}(\mu_{1}, \mu_{2})
        \]
        which contradicts $\mathcal{W}_{p}(\mu_{1},\mu_{2}) = 0$. With this, we have that $\forall~ u\in \mathcal{C}^{0}_{b}(\Omega)$
        \[
            \int\limits_{\Omega} u(x) ~d\mu_{1}(x) = \int\limits_{\Omega\times\Omega} u(x) ~d\pi(x,y) = \int\limits_{\Omega\times\Omega} u(y) ~d\pi(x,y) = \int\limits_{\Omega} u(y) ~d\mu_{2}(y)
        \]
        and thus we have $\mu_{1} = \mu_{2}$ by Theorem \ref{Ch1:EqualMeasuresViaBCF}.\\
        \indent Now let $\mu_{i}\in\mathscr{P}(\mathcal{F}) ~(i=1,2,3), \pi_{12}\in\mathbb{K}(\mu_{1},\mu_{2}), \pi_{23}\in\mathbb{K}(\mu_{2},\mu_{3})$, and, by the Lemma \ref{Ch4:Gluing}, $\pi_{123}\in\mathscr{P}(\mathcal{F}^{\otimes 3})$ coupling $\pi_{12}$ and $\pi_{23}$. Letting $\pi_{13}(\bullet,\bullet) = \pi_{123}(\bullet,\Omega,\bullet)$ (not necessarily optimal), we obtain
        \begin{align*}
            \mathcal{W}_{p}(\mu_{1},\mu_{3}) 
                &\leqq \left(~ \int\limits_{\Omega\times\Omega} d^{p}(x_{1},x_{2}) ~d\pi_{13} \right)^{\frac{1}{p}}\\
                & = \left(~\int\limits_{\Omega^{\otimes 3}} d^{p}(x_{1},x_{3}) ~d\pi_{123} \right)^{\frac{1}{p}}\\
                & \leqq \left(~\int\limits_{\Omega^{\otimes 3}} \left[~ d(x_{1},x_{2}) + d(x_{2},x_{3}) ~\right]^{p} ~d\pi_{123} \right)^{\frac{1}{p}}\\
                & \leqq \left( \int\limits_{\Omega^{\otimes 3}} d^{p}(x_{1},x_{2}) ~d\pi_{123} \right)^{\frac{1}{p}} + \left(~ \int\limits_{\Omega^{\otimes 3}} d^{p}(x_{2},x_{3}) ~d\pi_{123} \right)^{\frac{1}{p}}\\
                & = \left(~ \int\limits_{\Omega\times\Omega} d^{p}(x_{1},x_{2}) ~d\pi_{12} \right)^{\frac{1}{p}} + \left(~ \int\limits_{\Omega\times\Omega} d^{p}(x_{2},x_{3}) ~d\pi_{23} \right)^{\frac{1}{p}}\\
                & = \mathcal{W}_{p}(\mu_{1},\mu_{2}) + \mathcal{W}_{p}(\mu_{2},\mu_{3})
        \end{align*}
        which proves the triangle inequality.
    \end{proof}

\break

\end{document}